\documentclass[numbook,envcountsame,smallcondensed]{amsart}

\usepackage{amssymb,amsmath,amsfonts,mathrsfs,cite,mathptmx,url}

\DeclareMathOperator{\alf}{alph}
\DeclareMathOperator{\simple}{sim}
\DeclareMathOperator{\mul}{mul}
\DeclareMathOperator{\occ}{occ}
\DeclareMathOperator{\var}{var}

\newtheorem{theorem}{Theorem}[section]
\newtheorem{proposition}[theorem]{Proposition}
\newtheorem{lemma}[theorem]{Lemma}
\newtheorem{corollary}[theorem]{Corollary}

\theoremstyle{definition}

\numberwithin{equation}{section}

\makeatletter

\renewcommand*\subjclass[2][2010]{\def\@subjclass{#2}\@ifundefined{subjclassname@#1}{\ClassWarning{\@classname}{Unknown edition (#1) of Mathematics Subject Classification; using '2010'.}}{\@xp\let\@xp\subjclassname\csname subjclassname@#1\endcsname}}

\renewcommand{\subjclassname}{\textup{2010} Mathematics Subject Classification}

\makeatother

\begin{document}

\title{Finiteness conditions for lattices of monoid varieties}
\thanks{Supported by the Ministry of Science and Higher Education of the Russian Federation (project FEUZ-2023-0022).}

\author{Sergey V. Gusev}

\address{Ural Federal University, Institute of Natural Sciences and Mathematics, Lenina 51, Ekaterinburg 620000, Russia}

\email{sergey.gusb@gmail.com}

\begin{abstract}
We classify all varieties of aperiodic monoids with central idempotents whose subvariety lattice is finite or satisfies the descending chain condition or satisfies the ascending chain condition. 
It turns out that for varieties in this class, the properties of having a finite subvariety lattice and a subvariety lattice satisfying the ascending chain condition are equivalent, and thus the property of having a subvariety lattice satisfying the ascending chain condition implies the one of having a subvariety lattice satisfying the descending chain condition.
\end{abstract}



\maketitle

\section{Background and overview}
\label{Sec: introduction}

In 1971, in his famous survey~\cite{Evans-71}, Trevor Evans posed the problem of classifying varieties of semigroups that are \textit{small} in the sense their lattice of subvarieties is finite. 
This problem, along with the problems of describing varieties whose lattice of subvarieties satisfies some other finiteness conditions, has attracted the attention of many authors. 
Among the finiteness conditions, the most attention have been received by the ascending chain condition and the descending chain condition.
For brevity, we will say that a variety $\mathbf V$ \textit{satisfies the ACC} [respectively, \textit{DCC}] if  its lattice $\mathfrak L(\mathbf V)$ of subvarieties satisfies the ascending chain condition [respectively, descending chain condition]. 
Nevertheless, despite significant progress, the problems of describing small semigroup varieties or semigroup varieties satisfying the ACC or the DCC still seem to be very far from a complete solution (see Lev Shevrin at al.~\cite[Section~10]{Shevrin-Vernikov-Volkov-09} for more details).

The present article is concerned with the finiteness conditions for lattices of varieties of \textit{monoids}, i.e., semigroups with an identity element.
Even though monoids are very similar to semigroups, the story turns out to be very different and also difficult.
Varieties of monoids whose subvariety lattice satisfies one of the three discussed finiteness conditions have not been systematically examined earlier, although a lot of non-trivial examples of such varieties have been known (see Section~7 in the recent survey~\cite{Gusev-Lee-Vernikov-22}).

The present paper is the first attempt at a systematic study of finiteness conditions for lattices of monoids varieties.
In view of the result by Pavel Kozhevnikov~\cite{Kozhevnikov-12}, there exist uncountably many group varieties whose subvariety lattice is isomorphic to just a 3-element chain.
Therefore, a classification of small monoid varieties (as well as monoid varieties satisfying the ACC or the DCC) is hardly possible in general, since such a classification must include a description of small varieties of periodic groups.
Thus, first of all, it is logical to focus specifically on the class of \textit{aperiodic} monoids, i.e., monoids that do not contain non-trivial subgroups.
However, experience suggests that even in the class of aperiodic monoids the discussed problems remain quite complex.
So, it is natural at first to try solving these problems within some subclass of the class of all aperiodic monoids. 

The class $\mathbf A_\mathsf{cen}$ of aperiodic monoids with central idempotents is a natural candidate.
This class is quite wide. 
It includes, in particular, all \textit{nilpotent} monoids, that is, monoids obtained from nilsemigroups by adjoining a new identity element.
The class $\mathbf A_\mathsf{cen}$ is not a variety, but it forms a so-called \textit{generalized variety}, i.e., a class of monoids closed under the formation of homomorphic images, subalgebras, finitary direct products and direct powers (see~\cite[Section~3.2]{Almeida-94}, for instance). 
Its ``finite'' part, i.e., all its finite members, form a pseudovariety of finite monoids, which was comprehensively studied by Howard Straubing~\cite{Straubing-82}. 
In particular, he showed that the pseudovariety of finite aperiodic moniods with central idempotents is the smallest pseudovariety containing the class of all finite nilpotent monoids as well as he described the variety of regular languages corresponding to this pseudovariety.

Subvarieties of $\mathbf A_\mathsf{cen}$ have been intensively studied for two last decades. 
This class is rich in examples of varieties interesting from specific points of view (see Daniel Glasson~\cite{Glasson-24a,Glasson-24b}, Sergey Gusev~\cite{Gusev-19,Gusev-24}, Sergey Gusev and Edmond Lee~\cite{Gusev-Lee-20}, Marcel Jackson~\cite{Jackson-05,Jackson-15}, Marcel Jackson and Edmond Lee~\cite{Jackson-Lee-18}, Marcel Jackson and Wen Ting Zhang~\cite{Jackson-Zhang-21}).
Besides that, one managed to completely describe subvarieties of $\mathbf A_\mathsf{cen}$ with some natural and important properties: hereditary finitely based varieties (Edmond Lee~\cite{Lee-12}), almost Cross varieties (Edmond Lee~\cite{Lee-13}), inherently non-finitely generated varieties (Edmond Lee~\cite{Lee-14}), and varieties with a distributive lattice of subvarieties (Sergey Gusev~\cite{Gusev-23}).
The main results of the present paper naturally fit into this series of results.

We completely classify small subvarieties of $\mathbf A_\mathsf{cen}$, subvarieties of $\mathbf A_\mathsf{cen}$ satisfying the ACC, and subvarieties of $\mathbf A_\mathsf{cen}$ satisfying the DCC.
We provide equational descriptions of such varieties. 
Namely, for each of the three discussed finiteness conditions, we present four countably infinite series of varieties such that every subvariety of $\mathbf A_\mathsf{cen}$ satisfying the given finiteness condition is contained in one of the varieties in these series. 
In particular, it turns out that for subvarieties of $\mathbf A_\mathsf{cen}$, the properties of being small and satisfying the ACC are equivalent.
As a corollary, any subvariety $\mathbf A_\mathsf{cen}$ satisfying the ACC satisfies also the DCC.

This paper is structured as follows.
In Section~\ref{sec: main result}, we formulate and discuss our main result.
Some background results are first given in Section~\ref{Sec: preliminaries}.   
In Section~\ref{sec: auxiliary}, we formulate and prove several auxiliary facts.
Finally, Section~\ref{sec: proof} is devoted to the proof of our main results.

\section{Our main result}
\label{sec: main result}

Let us briefly recall a few notions that we need to formulate our main result.
Let $\mathcal X$ be a countably infinite set called an \textit{alphabet}. 
As usual, let~$\mathcal X^\ast$ denote the free monoid over the alphabet~$\mathcal X$. 
Elements of~$\mathcal X$ are called \textit{letters} and elements of~$\mathcal X^\ast$ are called \textit{words}.
We treat the identity element of~$\mathcal X^\ast$ as \textit{the empty word}, which is denoted by~$1$.  
Words and letters are denoted by small Latin letters. 
However, words unlike letters are written in bold. 
An identity is written as $\mathbf u \approx \mathbf v$, where $\mathbf u,\mathbf v \in \mathcal X^\ast$; it is \textit{non-trivial} if $\mathbf u \ne \mathbf v$.

As usual, $\mathbb N$ denote the set of all natural numbers. 
Let $\mathbb N_0:=\mathbb N\cup\{0\}$. 
For any $n\in\mathbb N$, we denote by $S_n$ the full symmetric group on the set $\{1,\dots,n\}$.  
For convenience, we put $S_0:=S_1$. 
For any $n,m,k\in\mathbb N_0$, $\rho\in S_{n+m}$ and $\tau\in S_{n+m+k}$, we define the words:
\begin{align*}
\mathbf a_{n,m}[\rho]&:=\biggl(\prod_{i=1}^n z_it_i\biggr)x\biggl(\prod_{i=1}^{n+m} z_{i\rho}\biggr)x\biggl(\prod_{i=n+1}^{n+m} t_iz_i\biggr),\\
\mathbf a_{n,m}^\prime[\rho]&:=\biggl(\prod_{i=1}^n z_it_i\biggr)\biggl(\prod_{i=1}^{n+m} z_{i\rho}\biggr)x^2\biggl(\prod_{i=n+1}^{n+m} t_iz_i\biggr),
\end{align*}
\begin{align*}
\mathbf a_{n,m}^{\prime\prime}[\rho]&:=\biggl(\prod_{i=1}^n z_it_i\biggr)x^2\biggl(\prod_{i=1}^{n+m} z_{i\rho}\biggr)\biggl(\prod_{i=n+1}^{n+m} t_iz_i\biggr),\\
\mathbf c_{n,m,k}[\tau]&:=\biggl(\prod_{i=1}^n z_it_i\biggr)xyt\biggl(\prod_{i=n+1}^{n+m} z_it_i\biggr)x\biggl(\prod_{i=1}^{n+m+k} z_{i\tau}\biggr)y\biggl(\prod_{i=n+m+1}^{n+m+k} t_iz_i\biggr).
\end{align*}
Let $\mathbf c_{n,m,k}^\prime[\tau]$ denote the word obtained from $\mathbf c_{n,m,k}[\tau]$ by swapping the first occurrences of $x$ and $y$.
We denote also by $\mathbf d_{n,m,k}[\tau]$ and $\mathbf d_{n,m,k}^\prime[\tau]$ the words obtained from the words $\mathbf c_{n,m,k}[\tau]$ and $\mathbf c_{n,m,k}^\prime[\tau]$, respectively, when reading the last words from right to left.
We fix notation for the following three identities:
\begin{align*}
\sigma_1:\enskip xyzxty\approx yxzxty,\ \ \ \ 
\sigma_2:\enskip xzytxy\approx xzytyx,\ \ \ \ 
\sigma_3:\enskip xzxyty\approx xzyxty.
\end{align*}
For any $n\in\mathbb N$, we define:
\[
\begin{aligned}
\omega_n&:\enskip x\biggl(\prod_{i=1}^nt_ix\biggr)\approx x^{n+1}\biggl(\prod_{i=1}^nt_i\biggr).
\end{aligned}
\]
Let
\[
\begin{aligned}
&\Phi_n:=\{x^n\approx x^{n+1},\,x^ny\approx yx^n\},\ n\in\mathbb N,\\
&\Psi_1:=\left\{
\mathbf a_{k,\ell}[\rho] \approx \mathbf a_{k,\ell}^\prime[\rho]
\mid
k,\ell\in\mathbb N,\,
\rho\in S_{k+\ell}
\right\},\\
&\Psi_2:=\left\{
\mathbf a_{k,\ell}[\rho] \approx \mathbf a_{k,\ell}^{\prime\prime}[\rho]
\mid
k,\ell\in\mathbb N,\,
\rho\in S_{k+\ell}
\right\},\\
&\Psi_3:=\left\{
\mathbf c_{k,\ell,m}[\rho]\approx\mathbf c_{k,\ell,m}^\prime[\rho],\,\mathbf d_{k,\ell,m}[\rho]\approx\mathbf d_{k,\ell,m}^\prime[\rho]
\mid
k,\ell,m\in\mathbb N,\,
\rho\in S_{k+\ell+m}
\right\}
.
\end{aligned}
\]
Let $\var\,\Sigma$ denote the monoid variety given by a set $\Sigma$ of identities. 
For brevity, we denote by $\mathbf V\Sigma$ the subvariety of a variety $\mathbf V$ defined by a set $\Sigma$ of identities.
For any $n \in \mathbb N$, put 
\[
\mathbf P_n:=\var\left\{\Phi_n,\,\Psi_1,\,\Psi_3\right\},\ 
\mathbf Q_n:=\var\left\{\Phi_n,\sigma_2,\sigma_3\right\},\ 
\mathbf R_n:=\mathbf P_n\{\omega_n\}
\ 
\text{ and } \  
\mathbf S_n:=\mathbf Q_n\{\omega_n\}.
\]
If $\mathbf V$ is a monoid variety, then $\mathbf V^\delta$ denotes the variety \textit{dual} to $\mathbf V$, i.e., the variety consisting of monoids anti-isomorphic to monoids from $\mathbf V$. Notice that 
\[
\mathbf P_n^\delta=\var\left\{\Phi_n,\,\Psi_2,\,\Psi_3\right\},\ 
\mathbf Q_n^\delta=\var\left\{\Phi_n,\sigma_1,\sigma_3\right\},\ 
\mathbf R_n^\delta=\mathbf P_n^\delta\{\omega_n\}
\ 
\text{ and } \  
\mathbf S_n^\delta=\mathbf Q_n^\delta\{\omega_n\}.
\]

Our first main result is the following

\begin{theorem}
\label{T: small=ACC}
For a subvariety $\mathbf V$ of $\mathbf A_\mathsf{cen}$ the following are equivalent:
\begin{itemize}
\item[\textup{(i)}] $\mathbf V$ is small;
\item[\textup{(ii)}] $\mathbf V$ satisfies the ACC;
\item[\textup{(iii)}] $\mathbf V$ is contained in  one of the varieties $\mathbf R_n$, $\mathbf R_n^\delta$, $\mathbf S_n$ or $\mathbf S_n^\delta$ for some $n\in\mathbb N$.
\end{itemize}
\end{theorem}

Our second main result is the following

\begin{theorem}
\label{T: DCC}
A subvariety of $\mathbf A_\mathsf{cen}$ satisfies the DCC if and only if it is contained in one of the varieties $\mathbf P_n$, $\mathbf P_n^\delta$, $\mathbf Q_n$ or $\mathbf Q_n^\delta$ for some $n\in\mathbb N$.
\end{theorem}

Theorems~\ref{T: small=ACC} and~\ref{T: DCC} directly imply the following 

\begin{corollary}
\label{C: ACC=>DCC}
If a subvariety of $\mathbf A_\mathsf{cen}$ satisfies the ACC, then it also satisfies the DCC.\qed
\end{corollary}

Notice that there is still no known example of a monoid variety that satisfies the ACC but violates the DCC (see~\cite[Question~7.8]{Gusev-Lee-Vernikov-22}).
In the context of semigroup varieties, such an example was provided by Mark Sapir~\cite{Sapir-91}. 
To the best of our knowledge, the article~\cite{Sapir-91} is the only work that contains examples of semigroup varieties of this kind.

\section{Preliminaries}
\label{Sec: preliminaries}

\subsection{Words, identities, deduction}

The \textit{alphabet} of a word $\mathbf w$, i.e., the set of all letters occurring in $\mathbf w$, is denoted by $\alf(\mathbf w)$. 
For a word $\mathbf w$ and a letter $x$, let $\occ_x(\mathbf w)$ denote the number of occurrences of $x$ in $\mathbf w$.
A letter $x$ is called \textit{simple} [\textit{multiple}] \textit{in a word} $\mathbf w$ if $\occ_x(\mathbf w)=1$ [respectively, $\occ_x(\mathbf w)>1$]. 
The set of all simple [multiple] letters in a word $\mathbf w$ is denoted by $\simple(\mathbf w)$ [respectively, $\mul(\mathbf w)$]. 
If $\mathbf w$ is a word and $\mathcal Z\subseteq\alf(\mathbf w)$, then we denote by $\mathbf w_{\mathcal Z}$ [respectively, $\mathbf w(\mathcal Z)$] the word obtained from $\mathbf w$ by removing all occurrences of letters from $\mathcal Z$ [respectively, $\alf(\mathbf w)\setminus \mathcal Z$]. 
If $\mathcal Z=\{z\}$, then we write $\mathbf w_z$ rather than $\mathbf w_{\{z\}}$.
The expression $_{i\mathbf w}x$ means the $i$th occurrence of a letter $x$ in a word $\mathbf w$. 
If the $i$th occurrence of $x$ precedes the $j$th occurrence of $y$ in a word $\mathbf w$, then we write $({_{i\mathbf w}x}) < ({_{j\mathbf w}y})$.

A variety $\mathbf V$ \textit{satisfies} an identity $\mathbf u \approx \mathbf v$, if for any monoid $M\in \mathbf V$ and any substitution $\varphi\colon \mathcal X \to M$, the equality $\varphi(\mathbf u)=\varphi(\mathbf v)$ holds in $M$.
An identity $\mathbf u \approx \mathbf v$ is \textit{directly deducible} from an identity $\mathbf s \approx \mathbf t$ if there exist some words $\mathbf a,\mathbf b \in \mathcal X^\ast$ and substitution $\varphi\colon \mathcal X \to \mathcal X^\ast$ such that $\{ \mathbf u, \mathbf v \} = \{ \mathbf a\varphi(\mathbf s)\mathbf b,\mathbf a\varphi(\mathbf t)\mathbf b \}$.
A non-trivial identity $\mathbf u \approx \mathbf v$ is \textit{deducible} from a set $\Sigma$ of identities if there exists some finite sequence $\mathbf u = \mathbf w_0, \dots, \mathbf w_m = \mathbf v$ of words such that each identity $\mathbf w_i \approx \mathbf w_{i+1}$ is directly deducible from some identity in $\Sigma$.

\begin{proposition}[Birkhoff's Completeness Theorem for Equational Logic; see {\cite[Theorem~1.4.6]{Almeida-94}}]
\label{P: deduction}
A variety $\var\,\Sigma$ satisfies an identity $\mathbf u \approx \mathbf v$ if and only if $\mathbf u \approx \mathbf v$ is deducible from $\Sigma$.\qed
\end{proposition}

\subsection{Rees quotient monoids}
\label{Subsec: Rees quotient monoids}

The following construction was used by Perkins~\cite{Perkins-69} to build the first example of a finite semigroup generating non-finitely based variety.
For any set $\mathcal W$ of words, let $M(\mathcal W)$ denote the Rees quotient monoid of $\mathcal X^\ast$ over the ideal of all words that are not subwords of any word in $\mathcal W$.

Given a variety $\mathbf V$, a word $\mathbf u$ is called an \emph{isoterm for} $\mathbf V$ if the only word $\mathbf v$ such that $\mathbf V$ satisfies the identity $\mathbf u \approx \mathbf v$ is the word $\mathbf u$ itself.

\begin{lemma}[\mdseries{\!\cite[Lemma~3.3]{Jackson-05}}]
\label{L: M(W) in V}
Let $\mathbf V$ be a monoid variety and $\mathcal W$ a set of words. 
Then $M(\mathcal W)\in\mathbf V$ if and only if each word in $\mathcal W$ is an isoterm for $\mathbf V$.\qed
\end{lemma}

Given a set $\mathcal W$ of words, let $\mathbf M(\mathcal W)$ denote the monoid variety generated by $M(\mathcal W)$.
For brevity, if $\mathbf w_1,\dots,\mathbf w_k\in\mathcal X^\ast$, then we write $M(\mathbf w_1,\dots,\mathbf w_k)$ [respectively, $\mathbf M(\mathbf w_1,\dots,\mathbf w_k)$] rather than $M(\{\mathbf w_1,\dots,\mathbf w_k\})$ [respectively, $\mathbf M(\{\mathbf w_1,\dots,\mathbf w_k\})$].

\subsection{Some known results}

The following fact is obvious.

\begin{lemma}
\label{L: x^n is an isoterm}
Let $\mathbf V$ be a monoid variety and $n\in\mathbb N$. 
Then $x^n$ is not an isoterm for $\mathbf V$ if and only if $\mathbf V$ satisfies the identity $x^n\approx x^m$ for some $m>n$.\qed
\end{lemma}

The following statement was established in the proof of Lemma~3.5 in~\cite{Gusev-Vernikov-21}.

\begin{lemma}
\label{L: swapping in linear-balanced}
Let $\mathbf V$ be a monoid variety such that $M(xt_1x\cdots t_nx)\in\mathbf V$. 
If $M(\mathbf p\,xy\,\mathbf q)\notin\mathbf V$, where $\mathbf p:=a_1t_1\cdots a_kt_k$ and $\mathbf q:=t_{k+1}a_{k+1}\cdots t_{k+\ell}a_{k+\ell}$ for some $k,\ell\in\mathbb N_0$ and $a_1,\dots,a_{k+\ell}$ are letters such that $\{a_1,\dots,a_{k+\ell}\}=\{x,y\}$ and $\occ_x(\mathbf p\mathbf q),\occ_y(\mathbf p\mathbf q)\le n$, then $\mathbf V$ satisfies the identity $\mathbf p\,xy\,\mathbf q\approx\mathbf p\,yx\,\mathbf q$.\qed
\end{lemma}

\section{Several auxiliary results}
\label{sec: auxiliary}

\subsection{Certain non-small varieties}
\label{subsec: non-small}

Let 
\[
\hat{\mathbb N}_0^2 :=\{(n,m)\in \mathbb N_0\times \mathbb N_0\mid |n-m|\le 1\}.
\]
For $n,m\in\mathbb N_0$, a permutation $\rho$ from $S_{n+m}$ is an $(n,m)$-\textit{permutation} if, for all $i=1,2,\dots,n+m-1$, one of the following holds:
\begin{itemize}
\item $1\le i\rho\le n$ and $n<(i+1)\rho\le n+m$;
\item $1\le (i+1)\rho\le n$ and $n<i\rho\le n+m$.
\end{itemize}
Evidently, if $\rho$ is a $(n,m)$-permutation, then $(n,m)\in \hat{\mathbb N}_0^2$.
The set of all $(n,m)$-permutations is denoted by $S_{n,m}$.

\begin{proposition}
\label{P: non-small}
The following varieties violate both the ACC and the DCC:
\begin{itemize}
\item[\textup{(i)}] $\mathbf M(\mathbf a_{n,m}[\rho])$, where $(n,m)\in \hat{\mathbb N}_0^2$ and $\rho\in S_{n,m}$;
\item[\textup{(ii)}] $\mathbf M(\mathbf c_{n,m,0}[\rho])$, where $n,m\in \mathbb N_0$ and $\rho\in S_{n+m}$;
\item[\textup{(iii)}] $\mathbf M(\mathbf c_{n,m,n+m+1}[\rho])$, where $n,m\in \mathbb N_0$ and $\rho\in S_{n+m,n+m+1}$.
\end{itemize}
\end{proposition}

\begin{proof}
It follows from the proofs of Propositions~3.4 and~3.5 in~\cite{Gusev-23} that the varieties $\mathbf M(\mathbf a_{n,m}[\rho])$ and $\mathbf M(\mathbf c_{n,m,n+m+1}[\rho])$ violate the DCC, while the variety $\mathbf M(\mathbf c_{n,m,0}[\rho])$ does not satisfy the DCC by Lemma~3.4 in~\cite{Gusev-Vernikov-21}.
So, it remains to verify that these varieties violate the ACC.

\smallskip

\noindent\textit{The variety $\mathbf M(\mathbf a_{n,m}[\rho])$ violates the ACC}.
Since $\rho$ is is an $(n,m)$-permutation, there are four possibilities:
\begin{itemize}
\item $n=m$ and $1\le 1\rho\le n$;
\item $n=m$ and $n+1\le 1\rho\le 2n$;
\item $n=m+1$ and so $1\le 1\rho\le n$;
\item $n=m-1$ and so $n+1\le 1\rho\le n+m$.
\end{itemize}
We will consider only the first possibility because the other ones are considered quite analogous.
In this case, $1\le i\rho\le n$ and $n+1\le (i+1)\rho\le 2n$ for any $i=1,3,\dots,2n-1$.

It is shown in the proof of Propositions~3.4 in~\cite{Gusev-23} that each variety of the form  variety $\mathbf M(\mathbf a_{n,m}[\rho])$ contains subvarieties of the same form with $n,m\ge2$.
This allows us to assume that $n,m\ge2$.
Now, for any $k\ge3$, put
\[
\mathbf a_k:=\biggl(\prod_{i=1}^k \mathbf p_{i}\biggr)x_1\mathbf r_1\biggl(\prod_{i=1}^{k-2}\hat{z}_{1\rho}^{(i+1)}x_{i+1}\mathbf r_{i+1}x_{i}\hat{z}_{(n+m)\rho}^{(i+1)}\biggr)\hat{z}_{1\rho}^{(k)}x_k\mathbf r_kx_{k-1}tx_k\biggl(\prod_{i=1}^k \mathbf q_{i}\biggr),
\]
where
\[
\mathbf p_j:=\biggl(\prod_{i=1}^n z_i^{(j)}t_i^{(j)}\hat{z}_i^{(j)}\hat{t}_i^{(j)}\biggr),\  
\mathbf q_j:=\biggl(\prod_{i=n+1}^{n+m} \hat{t}_i^{(j)}\hat{z}_i^{(j)}t_i^{(j)}z_i^{(j)}\biggr),\ 
\mathbf r_j:=\biggl(\prod_{i=1}^{n+m} z_{i\rho}^{(j)}\biggr),\ j\in\mathbb N.
\]
Since $\psi(\mathbf a_k)\mathbf q_{k+1}=\mathbf a_{k+1}$, where $\psi\colon \mathcal X\to \mathcal X^\ast$ is the substitution given by
\[
\psi(v):=
\begin{cases}
\hat{z}_{(n+m)\rho}^{(k)}\hat{z}_{1\rho}^{(k+1)}x_{k+1}\mathbf r_{k+1}&\text{if } v=t,\\
\hat{t}_n^{(k)}\mathbf p_{k+1}&\text{if } v=\hat{t}_n^{(k)},\\
tx_{k+1}\hat{t}_{n+1}^{(1)}&\text{if } v=\hat{t}_{n+1}^{(1)},\\
v&\text{otherwise},
\end{cases}
\]
Lemma~\ref{L: M(W) in V} implies that
\begin{equation}
\label{inclusions M(a_{n,m}[rho])}
\mathbf M(\mathbf a_3)\subseteq\mathbf M(\mathbf a_4)\subseteq\cdots\subseteq\mathbf M(\mathbf a_k)\subseteq\cdots
\end{equation}

We need the following auxiliary result.

\begin{lemma}
\label{L: a_X isoterm}
Let $\mathbf V$ be a variety such that the word $xzxyty$ is an isoterm for $\mathbf V$.
Then the word $(\mathbf a_k)_{\{x_1,\dots,x_{k-1}\}}$ is an isoterm for $\mathbf V$ for any $k\ge3$.
\end{lemma}

\begin{proof}
Let $(\mathbf a_k)_{\{x_1,\dots,x_{k-1}\}}\approx \mathbf a^\prime$ be an identity of $\mathbf V$.
Clearly, the word $xyx$ is an isoterm for $\mathbf V$. 
Hence 
\[
\mathbf a^\prime=\biggl(\prod_{i=1}^k \mathbf p_{i}\biggr)\mathbf htx_k\biggl(\prod_{i=1}^k \mathbf q_{i}\biggr),
\]
where $\mathbf h$ is a word with
\[
\simple(\mathbf h)=\alf(\mathbf h)=\alf(\mathbf r_1\cdots\mathbf r_k)\cup\{x_k,z_{1\rho}^{(2)},z_{(n+m)\rho}^{(2)},\dots,z_{1\rho}^{(k-1)},z_{(n+m)\rho}^{(k-1)},z_{1\rho}^{(k)}\}.
\]
It follows that if $\mathbf h$ does not coincide with the subword of $(\mathbf a_k)_{\{x_1,\dots,x_{k-1}\}}$ lying between $\hat{t}_n^{(k)}$ and $\hat{t}_{n+1}^{(1)}$, then the identity $(\mathbf a_k)_{\{x_1,\dots,x_{k-1}\}}\approx \mathbf a^\prime$ implies the identity $\sigma_3$, contradicting the condition of the lemma.
Therefore, $\mathbf a^\prime=(\mathbf a_k)_{\{x_1,\dots,x_{k-1}\}}$ and so $(\mathbf a_k)_{\{x_1,\dots,x_{k-1}\}}$ is an isoterm for $\mathbf V$.
\end{proof}

Consider an arbitrary identity of the form $\mathbf a_k\approx\mathbf a$ that is satisfied by $M(\mathbf a_{n,m}[\rho])$. 
Since the word $xzxyty$ is an isoterm for $\mathbf M(\mathbf a_{n,m}[\rho])$, it follows from Lemma~\ref{L: a_X isoterm} that $\mathbf a_{\{x_1,\dots,x_{k-1}\}}=(\mathbf a_k)_{\{x_1,\dots,x_{k-1}\}}$.
Since the word $\mathbf a_k(\mathcal T_j)$ coincides (up to renaming of letters) with the word $\mathbf a_{n,m}[\rho]$, where
\[
\mathcal T_j:=\{x_1,z_i^{(j)},t_i^{(j)}\mid 1\le i\le n+m\}, \ j=1,2,
\]
Lemma~\ref{L: M(W) in V} implies that $\mathbf a(\mathcal T_j)=\mathbf a_k(\mathcal T_j)$, whence $({_{1\mathbf a}x_1})<({_{2\mathbf a}z_{1\rho}^{(1)}})$ and $({_{1\mathbf a}z_{(n+m)\rho}^{(2)}})<({_{2\mathbf a}x_1})$.
Further, $({_{2\mathbf a}x_1})<({_{1\mathbf a}\hat{z}_{(n+m)\rho}^{(2)}})$ because the word $\mathbf a(\{x_1,\hat{z}_{(n+m)\rho}^{(2)},\hat{t}_{(n+m)\rho}^{(2)},z_i^{(2)},t_i^{(2)}\mid 1\le i\le n+m-1\})$ coincides (up to renaming of letters) with the word $\mathbf a_{n,m}[\rho]$ otherwise.
Finally, $({_{1\mathbf a}t_n^{(k)}})<({_{1\mathbf a}x_1})$ since $xyx$ is an isoterm for $\mathbf M(\mathbf a_{n,m}[\rho])$.
By a similar argument, we can show that
\[
\begin{aligned}
&({_{2\mathbf a}\hat{z}_{1\rho}^{(i)}})<({_{1\mathbf a}x_i})<({_{2\mathbf a}z_{1\rho}^{(i)}}),\ 
({_{1\mathbf a}z_{(n+m)\rho}^{(i+1)}})<({_{2\mathbf a}x_1})<({_{1\mathbf a}\hat{z}_{(n+m)\rho}^{(i+1)}}),\ i=2,\dots,k-2,\\
&({_{2\mathbf a}\hat{z}_{1\rho}^{(k-1)}})<({_{1\mathbf a}x_{k-1}})<({_{2\mathbf a}z_{1\rho}^{(k-1)}}),\ 
({_{1\mathbf a}z_{(n+m)\rho}^{(k)}})<({_{2\mathbf a}x_{k-1}})<({_{1\mathbf a}t}). 
\end{aligned}
\]
We see that $\mathbf a_k$ is an isoterm for $\mathbf M(\mathbf a_{n,m}[\rho])$.
Hence $M(\mathbf a_k)\in \mathbf M(\mathbf a_{n,m}[\rho])$ by Lemma~\ref{L: M(W) in V}. 

It remains to show that all inclusions in~\eqref{inclusions M(a_{n,m}[rho])} are strict.
To do this, verify that $M(\mathbf a_k)$ satisfies the identity $\mathbf a_{k+1}\approx x_1^2(\mathbf a_{k+1})_{x_1}$.
Consider an arbitrary substitution $\varphi\colon \mathscr X\to M(\mathbf a_k)$.
We are going to show that $\varphi(\mathbf a_{k+1})=\varphi(x_1^2(\mathbf a_{k+1})_{x_1})$.
If $\varphi({x_1})=1$, then the required claim is evident.
Let now $\varphi({x_1})\ne 1$.
Then $\varphi(x_1^2(\mathbf a_{k+1})_{x_1})=0$ because the word $\mathbf a_k$ is square-free.
Arguing by contradiction, suppose that $\varphi(\mathbf a_{k+1})$ is a subword of $\mathbf a_k$, i.e., there are $\mathbf a,\mathbf b\in\mathscr X^\ast$ such that $\mathbf a_k=\mathbf a\varphi(\mathbf a_{k+1})\mathbf b$.

Notice that every subword of length $2$ of $\mathbf a_k$ has the unique occurrence in $\mathbf a_k$.
Since each letter occurs in $\mathbf a_k$ at most twice, it follows that
\begin{itemize}
\item[\textup{($\ast$)}] for any $c\in\mul(\mathbf a_{k+1})$, either $\varphi(c)=1$ or $(\varphi({_{1\mathbf a_{k+1}}}c),\varphi({_{2\mathbf a_{k+1}}}c))=({_{1\mathbf a_k}}d,{_{2\mathbf a_k}}d)$ for some $d\in\mul(\mathbf a_k)$.
\end{itemize}
Then, since $\{x_1,\dots,x_{j-1}\}$ is the set of all multiple letters in $\mathbf a_j$ between the first and second occurrences of which there are no simple letters, it follows that $(\varphi({_{1\mathbf a_{k+1}}}x_1),\varphi({_{2\mathbf a_{k+1}}}x_1))=({_{1\mathbf a_{k+1}}}x_r,{_{2\mathbf a_{k+1}}}x_r)$ for some $r\in\{1,\dots,k-1\}$.
Then 
\[
\varphi(x_1\mathbf r_1\hat{z}_{1\rho}^{(2)}x_2\mathbf r_2x_1)=
\begin{cases}
x_1\mathbf r_1\hat{z}_{1\rho}^{(2)}x_2\mathbf r_2x_1&\text{if } r=1,\\
x_r\mathbf r_rx_{r-1}\hat{z}_{(n+m)\rho}^{(r)}\hat{z}_{1\rho}^{(r+1)}x_{r+1}\mathbf r_{r+1}x_r&\text{if } r=2,\dots,k-1.
\end{cases}
\]
By~($\ast$), $r=1$ and $(\varphi({_{1\mathbf a_{k+1}}}x_2),\varphi({_{2\mathbf a_{k+1}}}x_2))=({_{1\mathbf a_{k+1}}}x_2,{_{2\mathbf a_{k+1}}}x_2))$.
Then 
\[
\varphi(x_2\mathbf r_2x_1\hat{z}_{(n+m)\rho}^{(2)}\hat{z}_{1\rho}^{(3)}x_3\mathbf r_3x_2)=x_2\mathbf r_2x_1\hat{z}_{(n+m)\rho}^{(2)}\hat{z}_{1\rho}^{(3)}x_3\mathbf r_3x_2.
\]
Now~($\ast$) applies again, yielding that $(\varphi({_{1\mathbf a_{k+1}}}x_2),\varphi({_{2\mathbf a_{k+1}}}x_2))=({_{1\mathbf a_{k+1}}}x_2,{_{2\mathbf a_{k+1}}}x_2))$.
Continuing the process, we can show by induction that $(\varphi({_{1\mathbf a_{k+1}}}x_i),\varphi({_{2\mathbf a_{k+1}}}x_i))=({_{1\mathbf a_{k+1}}}x_i,{_{2\mathbf a_{k+1}}}x_i))$ for any $i=1,\dots,k$.
Hence 
\[
\varphi(x_k\mathbf r_kx_{k-1}\hat{z}_{(n+m)\rho}^{(k)}\hat{z}_{1\rho}^{(k+1)}x_{k+1}\mathbf r_{k+1}x_k)=x_k\mathbf r_kx_{k-1}tx_k,
\] 
and so $\varphi(\hat{z}_{(n+m)\rho}^{(k)}\hat{z}_{1\rho}^{(k+1)}x_{k+1}\mathbf r_{k+1})=t$, contradicting the fact that $t\in\simple(\mathbf a_k)$, while 
\[
\alf(\hat{z}_{(n+m)\rho}^{(k)}\hat{z}_{1\rho}^{(k+1)}x_{k+1}\mathbf r_{k+1})\subseteq\mul(\mathbf a_{k+1}).
\]
Thus, we have proved that $\mathbf a_{k+1}\approx x_1^2(\mathbf a_{k+1})_{x_1}$ is satisfied by $M(\mathbf a_k)$.
This implies that all the inclusions in~\eqref{inclusions M(a_{n,m}[rho])} are strict and, therefore, $\mathbf M(\mathbf a_{n,m}[\rho])$ violates the ACC.

\smallskip

\noindent\textit{The variety $\mathbf M(\mathbf c_{n,m,0}[\rho])$ violates the ACC}.
It is shown in the proof of Lemma~3.4 in~\cite{Gusev-Vernikov-21} that each variety of the form  variety $\mathbf M(\mathbf c_{n,m,0}[\rho])$ contains subvarieties of the same form with $n,m\ge2$.
This allows us to assume that $n,m\ge2$.
Now for any $k\in\mathbb N_0$, put
\[
\mathbf c_k:=\biggl(\prod_{i=0}^k \mathbf p_{2i}\biggr)x_0y_0\biggl(\prod_{i=1}^k \mathbf s_{2i-1}\biggr)s_{2k+1}x_{2k+1}s_{2k+1}^\prime y_{2k+1}t\biggl(\prod_{i=0}^k \mathbf q_{2i}\biggr)\biggl(\prod_{i=k}^0 \mathbf r_{2i}\biggr),
\]
where
\begin{align}
\label{p_j=}
&\mathbf p_j:=\biggl(\prod_{i=1}^n z_i^{(j)}t_i^{(j)}\biggr),\\
\label{q_j=}
&\mathbf q_j:=\biggl(\prod_{i=n+1}^{n+m} z_i^{(j)}t_i^{(j)}\biggr),\\
\notag
&\mathbf r_j:=s_jx_jz_{1\rho}^{(j)}x_{j+1}\biggl(\prod_{i=2}^{n+m-1} z_{i\rho}^{(j)}\biggr)y_{j+1}z_{(n+m)\rho}^{(j)}y_j,\\
\label{s_j=}
&\mathbf s_j:=s_jx_jx_{j+1}y_{j+1}y_j,\ j\in\mathbb N_0.
\end{align}
Since $\psi(\mathbf c_k)=\mathbf c_{k+1}$, where $\psi\colon \mathcal X\to \mathcal X^\ast$ is the substitution given by
\[
\psi(v):=
\begin{cases}
x_{2k+3}s_{2k+3} y_{2k+3}t&\text{if } v=t,\\
t_n^{(2k)}\mathbf p_{2k+2}&\text{if } v=t_n^{(2k)},\\
t_{n+m}^{(2k)}\mathbf q_{2k+2}&\text{if } v=t_{n+m}^{(2k)},\\
\mathbf r_{2k+2}s_{2k}&\text{if } v=s_{2k},\\
x_{2k+2}y_{2k+2}&\text{if } v=s_{2k+1}^\prime,\\
v&\text{otherwise},
\end{cases}
\]
Lemma~\ref{L: M(W) in V} implies that
\begin{equation}
\label{inclusions M(c_{n,m,0}[rho])}
\mathbf M(\mathbf c_0)\subseteq\mathbf M(\mathbf c_1)\subseteq\cdots\subseteq\mathbf M(\mathbf c_k)\subseteq\cdots
\end{equation}

We need the following auxiliary result which is a consequence of Lemma~2.5 in~\cite{Gusev-19} and the dual statement.

\begin{lemma}
\label{L: c_k=c}
Let $\mathbf V$ be a variety such that the words $xyzxty$ and $xzytxy$ are isoterms for $\mathbf V$.
If $\mathbf V$ satisfies an identity of the form $\mathbf c_k\approx \mathbf c$, then
\begin{equation}
\label{c=}
\mathbf c=\biggl(\prod_{i=0}^k \mathbf p_{2i}\biggr)\mathbf d_0\biggl(\prod_{i=0}^k \mathbf q_{2i}\biggr)\biggl(\prod_{i=1}^k x_{2i-1}\mathbf d_{2i}y_{2i-1}\biggr)x_{2k+1}s_{2k+1} y_{2k+1}\biggl(\prod_{i=k}^0 \mathbf r_{2i}\biggr),
\end{equation}
where $\mathbf d_{2j}\in\{x_{2j}y_{2j},y_{2j}x_{2j}\}$, $j=0,\dots,k$.\qed
\end{lemma}

If $M(\mathbf c_k)\notin \mathbf M(\mathbf c_{n,m,0}[\rho])$, then $M(\mathbf c_{n,m,0}[\rho])$ satisfies some non-trivial identity $\mathbf c_k\approx\mathbf c$ by Lemma~\ref{L: M(W) in V}. 
Since the words $xyzxty$ and $xzytxy$ are isoterms for $\mathbf M(\mathbf c_{n,m,0}[\rho])$, it follows from Lemma~\ref{L: c_k=c} that the equality~\eqref{c=} holds, where $\mathbf d_{2j}\in\{x_{2j}y_{2j},y_{2j}x_{2j}\}$, $j=0,\dots,k$.
Since the identity $\mathbf c_k\approx\mathbf c$ is non-trivial, there is $s\in\{0,\dots,k\}$ such that $\mathbf d_{2s}=y_{2s} x_{2s}$.
Then the identity $\mathbf c_k(\mathcal T)\approx\mathbf c(\mathcal T)$, where
\[
\mathcal T:=\{x_{2s},y_{2s},s_{2s-1},t_i^{(2s)},z_i^{(2s)}\mid 1\le i\le n+m\},
\] 
coincides (up to renaming of letters) with the identity $\mathbf c_{n,m,0}[\rho]\approx \mathbf c_{n,m,0}^\prime[\rho]$, contradicting the fact that $\mathbf M(\mathbf c_{n,m,0}[\rho])$ satisfies $\mathbf c_k\approx\mathbf c$.
Therefore, $\mathbf M(\mathbf c_k)\subseteq\mathbf M(\mathbf c_{n,m,0}[\rho])$ for any $k\in\mathbb N_0$.

It remains to show that all inclusions in~\eqref{inclusions M(c_{n,m,0}[rho])} are strict.
To do this, verify that $M(\mathbf c_k)$ satisfies the identity $\mathbf c_{k+1}\approx \mathbf c_{k+1}^\prime$, where $\mathbf c_{k+1}^\prime$ is the word obtained from $\mathbf c_{k+1}$ by swapping the first occurrences of $x_0$ and $y_0$.
Consider an arbitrary substitution $\varphi\colon \mathscr X\to M(\mathbf c_k)$.
We are going to show that $\varphi(\mathbf c_{k+1})=\varphi(\mathbf c_{k+1}^\prime)$.
We may assume that either $\varphi(\mathbf c_{k+1})$ or $\varphi(\mathbf c_{k+1}^\prime)$ is a subword of $\mathbf c_k$.
If $\varphi(x_0)=1$ or $\varphi(y_0)=1$, then the required claim is evident.
Let now $\varphi(x_0)\ne1$ and $\varphi(y_0)\ne1$.
Since the words $xyzxty$ and $xzytxy$ are isoterms for $\mathbf M(\mathbf c_k)$, Lemma~\ref{L: c_k=c} implies that $\{\varphi(x_0),\varphi(y_0)\}=\{x_{2r},y_{2r}\}$ for some $r\in\{0,\dots,k\}$.
Then $\varphi(\mathbf c_{k+1}^\prime)=0$ because the word $({_{1\mathbf c_k}}x_{2r})<({_{1\mathbf c_k}}y_{2r})<({_{2\mathbf c_k}}x_{2r})<({_{2\mathbf c_k}}y_{2r})$, while $({_{1\mathbf c_{k+1}^\prime}}y_0)<({_{1\mathbf c_{k+1}^\prime}}x_0)<({_{2\mathbf c_{k+1}^\prime}}x_0)<({_{2\mathbf c_{k+1}^\prime}}y_0)$.
Therefore, $\varphi(\mathbf c_{k+1})$ is a subword of $\mathbf c_k$, i.e., there are $\mathbf a,\mathbf b\in\mathscr X^\ast$ such that $\mathbf c_k=\mathbf a\varphi(\mathbf c_{k+1})\mathbf b$.
Then $(x_{2r},y_{2r})=(\varphi(x_0),\varphi(y_0))$ and so 
\[
\varphi\biggl(x_0z_{1\rho}^{(0)}x_{1}\biggl(\prod_{i=2}^{n+m-1} z_{i\rho}^{(0)}\biggr)y_{1}z_{(n+m)\rho}^{(0)}y_0\biggr)=x_{2r}z_{1\rho}^{({2r})}x_{2r+1}\biggl(\prod_{i=2}^{n+m-1} z_{i\rho}^{({2r})}\biggr)y_{{2r}+1}z_{(n+m)\rho}^{({2r})}y_{2r}
\]
Notice that every subword of length $2$ of $\mathbf c_k$ has the unique occurrence in $\mathbf c_k$.
Since each letter occurs in $\mathbf c_k$ at most twice, it follows that
\begin{itemize}
\item[\textup{($\ast$)}] for any $c\in\mul(\mathbf c_{k+1})$, either $\varphi(c)=1$ or $(\varphi({_{1\mathbf c_{k+1}}}c),\varphi({_{2\mathbf c_{k+1}}}c))=({_{1\mathbf c_k}}d,{_{2\mathbf c_k}}d)$ for some $d\in\mul(\mathbf c_k)$.
\end{itemize}
In particular, this fact implies that $(x_{2r+1},y_{2r+1})=(\varphi(x_1),\varphi(y_1))$.
If $r\le k-1$, then $\varphi(x_1x_2y_2y_1)=x_{2r+1}x_{2r+2}y_{2r+2}y_{2r+1}$.
Now~($\ast$) applies again, yielding that $(x_{2r+2},y_{2r+2})=(\varphi(x_2),\varphi(y_2))$.
Continuing this process, we can show by induction that $(x_{2r+i},y_{2r+i})=(\varphi(x_i),\varphi(y_i))$ for any $i=0,\dots,2k-2r+1$.
Hence 
\[
\varphi(x_{2k-2r+1}x_{2k-2r+2}y_{2k-2r+2}y_{2k-2r+1})=x_{2k+1}s_{2k+1}y_{2k+1}
\] 
and so $\varphi(x_{2k-2r+2}y_{2k-2r+2})=s_{2k+1}$, contradicting the fact that $s_{2k+1}\in\simple(\mathbf c_k)$, while $x_{2k-2r+2},y_{2k-2r+2}\in\mul(\mathbf c_{k+1})$.
Thus, we have proved that $\mathbf c_{k+1}\approx \mathbf c_{k+1}^\prime$ is satisfied by $M(\mathbf c_k)$.
This implies that all the inclusions in~\eqref{inclusions M(c_{n,m,0}[rho])} are strict and, therefore, $\mathbf M(\mathbf c_{n,m,0}[\rho])$ violates the ACC.

\smallskip

\noindent\textit{The variety $\mathbf M(\mathbf c_{n,m,n+m+1}[\rho])$ violates the ACC}.
It is shown in the proof of Proposition~3.4 in~\cite{Gusev-23} that each variety of the form  variety $\mathbf M(\mathbf c_{n,m,n+m+1}[\rho])$ contains infinitely many subvarieties of the same form.
This allows us to assume that $n,m\ge1$.
Now for any $k\in\mathbb N_0$, put
\[
\mathbf c_k:=\biggl(\prod_{i=0}^k \mathbf p_{2i}\biggr)x_0y_0\biggl(\prod_{i=1}^k \mathbf s_{2i-1}\biggr)s_{2k+1}x_{2k+1}s_{2k+1}^\prime y_{2k+1}t\biggl(\prod_{i=0}^k \mathbf q_{2i}\biggr)\biggl(\prod_{i=k}^0 \mathbf r_{2i}\biggr)\biggl(\prod_{i=0}^k \mathbf t_{2i}\biggr),
\]
where the words $\mathbf p_j$, $\mathbf q_j$ and $\mathbf s_j$ are defined by the formulas~\eqref{p_j=},~\eqref{q_j=} and~\eqref{s_j=}, respectively, and
\[
\begin{aligned}
&\mathbf r_j:=s_jx_jz_{1\rho}^{(j)}x_{j+1}x_{j+1}^\prime\biggl(\prod_{i=2}^{2n+2m} z_{i\rho}^{(j)}\biggr)y_{j+1}^\prime y_{j+1}z_{(2n+2m+1)\rho}^{(j)}y_j,\\
&\mathbf t_j:=\biggl(\prod_{i=n+m+1}^{2n+2m+1} t_i^{(j)}z_i^{(j)}\biggr)s_jx_jx_{j+1}y_{j+1}y_j,\ j\in\mathbb N_0.
\end{aligned}
\]
Using the fact that the words $xyzxty$ and $xzxyty$ are isoterms for $\mathbf M(\mathbf c_{n,m,n+m+1}[\rho])$ and $\mathbf M(\mathbf c_k)$, we can reason as in the proof of Item~(ii) and show that $\mathbf M(\mathbf c_k)\subseteq \mathbf M(\mathbf c_{n,m,n+m+1}[\rho])$, and all the inclusions in~\eqref{inclusions M(c_{n,m,0}[rho])} hold and are strict. 
The variety $\mathbf M(\mathbf c_{n,m,n+m+1}[\rho])$ thus violates the ACC.
We allow ourselves to omit the corresponding arguments because they are quite analogous to ones in the proof of Item~(ii).
\end{proof}

Let
\[
\mathbf N:=\var\{\Phi_2,\omega_2,\sigma_2,\sigma_3\}.
\]

\begin{corollary}
\label{C: non-small}
The variety $\mathbf M(xzxyty)\vee \mathbf N$ violates both the ACC and the DCC.
\end{corollary}

\begin{proof}
It is shown in the proof of Proposition~3.2 in~\cite{Gusev-23} that $M(\mathbf c_{0,0,1}[\varepsilon])\in \mathbf M(xzxyty)\vee \mathbf N$, where $\varepsilon\in S_1$.
Now Proposition~\ref{P: non-small}(iii) applies.
\end{proof}

\subsection{Identities defining varieties that satisfy $\Psi_1\cup\Psi_3$}
\label{subsec: identities defining varieties}

\begin{proposition}
\label{P: Psi_1,Psi_3 subvarieties}
Each variety satisfying $\Psi_1\cup\Psi_3$ can be defined by the identities in $\Psi_1\cup\Psi_3$ together with some of the following identities:
\begin{equation}
\label{one letter in a block}
x^{e_0}\biggl(\prod_{i=1}^r t_ix^{e_i}\biggr) \approx x^{f_0}\biggl(\prod_{i=1}^r t_ix^{f_i}\biggr),
\end{equation}
where $r,e_0,f_0,\dots,e_r,f_r\in\mathbb N_0$; and
\begin{equation}
\label{two letters in a block}
\biggl(\prod_{i=1}^k a_i^{g_i}t_i\biggr) x^py^q \biggl(\prod_{i=k+1}^{k+\ell} t_i a_i^{g_i}\biggr)\approx\biggl(\prod_{i=1}^k a_i^{g_i}t_i\biggr) y^qx^p  \biggl(\prod_{i=k+1}^{k+\ell} t_ia_i^{g_i}\biggr),
\end{equation}
where $k,\ell\in\mathbb N_0$, $p,q\in\mathbb N$, $g_1,\dots,g_{k+\ell}\in \mathbb N_0$ and $a_1,\dots,a_{k+\ell}\in\{x,y\}$.
\end{proposition}

To prove Proposition~\ref{P: Psi_1,Psi_3 subvarieties}, we need two auxiliary results and some definitions.

\begin{lemma}
\label{L: from pxqxr to pqxxr}
If $\mathbf w:=\mathbf px\mathbf qx\mathbf r$ and $\alf(\mathbf q)\subseteq\mul(\mathbf w)$, then the set $\Psi_1$ of identities implies the identity $\mathbf w\approx\mathbf p\mathbf qx^2\mathbf r$.
\end{lemma}

\begin{proof}
Two cases are possible.

\smallskip

\noindent\textit{Case }1: $x\notin \alf(\mathbf q)$. 
If $\mul(\mathbf q)=\emptyset$, then, since $\alf(\mathbf q)\subseteq\mul(\mathbf w)$, one can find $n,m\in\mathbb N_0$ and $\rho\in S_{n+m}$ such that the length of $\mathbf q$ equals $n+m$ and $\mathbf a_{n,m}[\rho]\approx \mathbf a_{n,m}^\prime[\rho]$ implies $\mathbf w\approx\mathbf p\mathbf qx^2\mathbf r$.
If $\mul(\mathbf q)\ne\emptyset$, then the similar argument shows that $\Psi_1$ implies the identities 
\[
\mathbf w=\mathbf px\mathbf qx\mathbf r\approx \mathbf p\mathbf q^{\prime}x\mathbf q^{\prime\prime}x\mathbf r\approx \mathbf p\mathbf q^{\prime}\mathbf q^{\prime\prime}x^2\mathbf r
\approx \mathbf p\mathbf qx^2\mathbf r,
\]
where $\mathbf q^{\prime}$ [respectively, $\mathbf q^{\prime\prime}$] is a word obtained from $\mathbf q$ by retaining the non-last [respectively, last] occurrence of each letter.

\smallskip

\noindent\textit{Case }2: $x\in \alf(\mathbf q)$. 
In this case, $\mathbf q$ can be represented as follows: $\mathbf q=\mathbf q_0x\mathbf q_1\cdots x\mathbf q_r$, where $r\in\mathbb N$ and $x\notin\alf(\mathbf q_0\cdots \mathbf q_r)$.
Then the argument from Case~1 shows that the identities 
\[
\mathbf w=\mathbf px\mathbf q_0\biggl(\prod_{i=1}^r x\mathbf q_i\biggr)x\mathbf r \approx \mathbf p\mathbf q_0x^2\mathbf q_1\biggl(\prod_{i=2}^r x\mathbf q_i\biggr)x\mathbf r\approx \dots\approx \mathbf p\mathbf q_0\biggl(\prod_{i=1}^r x\mathbf q_i\biggr)x^2\mathbf r=\mathbf p\mathbf qx^2\mathbf r
\]
follow from $\Psi_1$, as required.
\end{proof}

\begin{lemma}
\label{L: from pxyqxrys to pyxqxrys}
If $\mathbf w:=\mathbf pxy\mathbf qx\mathbf ry\mathbf s$ and $\alf(\mathbf r)\subseteq\mul(\mathbf w)$, then, for $j=1,2$, the set 
\begin{equation}
\label{set of identities}
\Psi_j\cup
\left\{
\mathbf c_{n,m,k}[\rho]\approx\mathbf c_{n,m,k}^\prime[\rho]
\mid
n,m,k\in\mathbb N,\,
\rho\in S_{n+m+k}
\right\}
\end{equation} 
of identities implies the identity $\mathbf w\approx\mathbf pyx\mathbf qx\mathbf ry\mathbf s$.
\end{lemma}

\begin{proof}
Clearly, we may assume without any loss that $x,y\notin\alf(\mathbf r)$.
If $\mul(\mathbf r)=\emptyset$, then, since $\alf(\mathbf r)\subseteq\mul(\mathbf w)$, one can find $n,m,k\in\mathbb N_0$ and $\rho\in S_{n+m+k}$ such that the length of $\mathbf r$ equals $n+m+k$ and $\mathbf c_{n,m,k}[\rho]\approx\mathbf c_{n,m,k}^\prime[\rho]$ implies $\mathbf w\approx\mathbf pyx\mathbf qx\mathbf ry\mathbf s$.
Let now $\mul(\mathbf r)\ne\emptyset$.
If $j=1$, then, in view of the above and Lemma~\ref{L: from pxqxr to pqxxr}, the set~\eqref{set of identities} implies the identities 
\[
\begin{aligned}
\mathbf w=\mathbf pxy\mathbf qx\mathbf ry\mathbf s\approx \mathbf pxy\mathbf q\mathbf r^{\prime}x\mathbf r^{\prime\prime}y\mathbf s\approx \mathbf pyx\mathbf q\mathbf r^{\prime}x\mathbf r^{\prime\prime}y\mathbf s\approx \mathbf pyx\mathbf qx\mathbf ry\mathbf s,
\end{aligned}
\]
where $\mathbf r^{\prime}$ [respectively, $\mathbf r^{\prime\prime}$] is a word obtained from $\mathbf r$ by retaining the non-last [respectively, last] occurrence of each letter.
If $j=2$, then, using the statement dual to Lemma~\ref{L: from pxqxr to pqxxr} instead of Lemma~\ref{L: from pxqxr to pqxxr}, we can show that the set~\eqref{set of identities} implies the identity $\mathbf w\approx\mathbf pyx\mathbf qx\mathbf ry\mathbf s$.
\end{proof}

Let $\mathbf w$ be a word with $\simple(\mathbf w)=\{t_1,\dots,t_m\}$. 
We may assume without loss of generality that $\mathbf w(t_1,\dots,t_m)=t_1\cdots t_m$. 
Then $\mathbf w=\mathbf w_0t_1\mathbf w_1\cdots t_m\mathbf w_m$ for some words $\mathbf w_0,\dots,\mathbf w_m$. 
The words $\mathbf w_0,\dots, \mathbf w_m$ are called \textit{blocks} of the word $\mathbf w$. 
The representation of the word $\mathbf w$ as a product of alternating simple letters and blocks is called \textit{decomposition} of $\mathbf w$.

If $\mathbf w \in \mathcal X^\ast$ and $x \in \alf(\mathbf w)$, then an \textit{island} formed by $x$ in $\mathbf w$ is a maximal subword of $\mathbf w$, which is a power of $x$. 

\begin{proof}[Proof of Proposition~\ref{P: Psi_1,Psi_3 subvarieties}]
We denote by $\Sigma$ the set of all identities of form~\eqref{one letter in a block}, where $r,e_0,f_0,\dots,e_r,f_r\in\mathbb N_0$, and of the form~\eqref{two letters in a block}, where $k,\ell\in\mathbb N_0$, $p,q\in\mathbb N$, $g_1,\dots,g_{k+\ell}\in \mathbb N_0$ and $a_1,\dots,a_{k+\ell}\in\{x,y\}$.
Let $\mathbf V$ be a subvariety of 
\[
\mathbf O:=\var\{\Psi_1,\,\Psi_3\}.
\]
Take an arbitrary identity $\mathbf u \approx \mathbf u^\prime$ of $\mathbf V$.
We need to verify that $\mathbf u \approx \mathbf u^\prime$ is equivalent within $\mathbf O$ to some subset of $\Sigma$.

If the variety $\mathbf V$ is commutative, then it follows from~\cite{Head-68} that $\mathbf u \approx \mathbf u^\prime$ is equivalent to a subset of $\{x^p\approx x^q,\,xy\approx yx\mid p,q\in\mathbb N\}\subseteq \Sigma$, and we are done. 
So, we may further assume that $\mathbf V$ is non-commutative.
In this case, $xy$ is an isoterm for $\mathbf V$ by Lemma~2.4 in~\cite{Gusev-23}.

Let $\mathbf u_0t_1\mathbf u_1\cdots t_m\mathbf u_m$ be the decomposition of $\mathbf u$. 
Since $xy$ is an isoterm for $\mathbf V$, the decomposition of $\mathbf u^\prime$ has the form $\mathbf u_0^\prime t_1\mathbf u_1^\prime\cdots t_m\mathbf u_m^\prime$.
According to Lemma~\ref{L: from pxqxr to pqxxr}, the identities in $\Psi_1$ can be used to convert the words $\mathbf u$ and $\mathbf u^\prime$ into some words $\mathbf v$ and $\mathbf v^\prime$, respectively, such that the following hold:
\begin{itemize}
\item the decompositions of $\mathbf v$ and $\mathbf v^\prime$ are of the form $\mathbf v_0t_1\mathbf v_1\cdots t_m\mathbf v_m$ and $\mathbf v_0^\prime t_1\mathbf v_1^\prime\cdots t_m\mathbf v_m^\prime$, respectively;
\item each of the words $\mathbf v_i$ and $\mathbf v_i^\prime$ has at most one island form by $x$ for any $x\in\mathcal X$ and $i=0,\dots,m$.
\end{itemize}
In view of this fact, it suffices to show that the identity $\mathbf v \approx \mathbf v^\prime$ is equivalent within $\mathbf O$ to some subset of $\Sigma$.
Further, the set 
\begin{equation}
\label{set of rigid identities}
\{\mathbf v(x,t_1,\dots,t_m)\approx \mathbf v^\prime(x,t_1,\dots,t_m)\mid x\in \mul(\mathbf v)=\mul(\mathbf v^\prime)\}
\end{equation}
of identities can be used to convert the words $\mathbf v$ and $\mathbf v^\prime$ into some words $\mathbf w$ and $\mathbf w^\prime$, respectively, such that the following hold:
\begin{itemize}
\item the decompositions $\mathbf w$ and $\mathbf w^\prime$ have the form $\mathbf w_0t_1\mathbf w_1\cdots t_m\mathbf w_m$ and $\mathbf w_0^\prime t_1\mathbf w_1^\prime\cdots t_m\mathbf w_m^\prime$, respectively;
\item each of the words $\mathbf w_i$ and $\mathbf w_i^\prime$ has at most one island form by $x$ for any $x\in\mathcal X$ and $i=0,\dots,m$;
\item $\occ_x(\mathbf w_i)=\occ_x(\mathbf w_i^\prime)$ for any $x\in\mathcal X$ and $i=0,\dots,m$.
\end{itemize}
Evidently, every identity from~\eqref{set of rigid identities} is of the form~\eqref{one letter in a block}.
Hence $\mathbf O\{\mathbf u \approx \mathbf u^\prime\}=\mathbf O\{\Gamma,\,\mathbf w\approx \mathbf w^\prime\}$ for some $\Gamma\subseteq\Sigma$.
In view of this fact, it remains to show that the identity $\mathbf w \approx \mathbf w^\prime$ is equivalent within $\mathbf O$ to a subset of $\Sigma$.

We call an identity $\mathbf c\approx\mathbf d$ 1-\textit{invertible} if $\mathbf c=\mathbf e^\prime\, xy\,\mathbf e^{\prime\prime}$ and $\mathbf d=\mathbf e^\prime\, yx\,\mathbf e^{\prime\prime}$ for some words $\mathbf e^\prime,\mathbf e^{\prime\prime}\in\mathcal X^\ast$ and letters $x,y\in\alf(\mathbf e^\prime\mathbf e^{\prime\prime})$. 
Let $j>1$. 
An identity $\mathbf c\approx\mathbf d$ is called $j$-\textit{invertible} if there is a sequence of words $\mathbf c=\mathbf w_0,\dots,\mathbf w_j=\mathbf d$ such that the identity $\mathbf w_i\approx\mathbf w_{i+1}$ is 1-invertible for each $i=0,\dots,j-1$ and $j$ is the least number with such a property. 
For convenience, we will call the trivial identity 0-\textit{invertible}. 

Notice that the identity $\mathbf w \approx \mathbf w^\prime$ is $r$-invertible for some $r\in\mathbb N_0$ because $\occ_x(\mathbf w_i)=\occ_x(\mathbf w_i^\prime)$ for any $x\in\mathcal X$ and $i=0,\dots,m$. 
We will use induction by $r$.

\smallskip

\textit{Induction base}. 
If $r=0$, then $\mathbf w=\mathbf w^\prime$, whence $\mathbf V\{\mathbf w\approx\mathbf w^\prime\}=\mathbf V\{\emptyset\}$.

\smallskip

\textit{Induction step}. 
Let $r>0$. 
Obviously, $\mathbf w_s\ne\mathbf w_s^\prime$ for some $s\in\{0,\dots,m\}$. 
Then there are letters $x$ and $y$ such that $\mathbf w_s=\mathbf a_s\, y^qx^p\,\mathbf b_s$ for some $p,q\in\mathbb N$ and $\mathbf a_s,\mathbf b_s\in\mathcal X^\ast$ with $x,y\notin\alf(\mathbf a_s\mathbf b_s)$, while the island $x^p$ precedes the island $y^q$ in $\mathbf w_s^\prime$. 
We denote by $\hat{\mathbf w}$ the word obtained from $\mathbf w$ by swapping of the islands $x^p$ and $y^q$ in the block $\mathbf w_s$. 

To complete the proof, it suffices to show that $\mathbf w \approx \hat{\mathbf w}$ holds in $\mathbf O\{\mathbf w \approx \mathbf w^\prime\}$ and $\mathbf O\{\mathbf w \approx \hat{\mathbf w}\}=\mathbf O\Gamma^\prime$ for some $\Gamma^\prime\subseteq\Sigma$.
Indeed, in this case, the identity $\hat{\mathbf w}\approx\mathbf w^\prime$ is \mbox{$(r-1)$}-invertible. 
By the induction assumption, $\mathbf O\{\hat{\mathbf w}\approx\mathbf w\}=\mathbf O\Gamma^{\prime\prime}$ for some $\Gamma^{\prime\prime}\subseteq\Sigma$, whence  
\[
\mathbf O\{\mathbf w\approx\mathbf w^\prime\}=\mathbf O\{\mathbf w\approx \hat{\mathbf w},\hat{\mathbf w}\approx \mathbf w^\prime\}=\mathbf O\{\Gamma^\prime,\,\Gamma^{\prime\prime}\},
\] and we are done.

If $x,y\in\alf(\mathbf w_{s^\prime})=\alf(\mathbf w_{s^\prime})$ for some $s^\prime\ne s$, then Lemma~\ref{L: from pxyqxrys to pyxqxrys} or the statement dual to it implies that $\mathbf O$ satisfies the identity $\mathbf w\approx\hat{\mathbf w}$. 
Thus, we may further assume that at most one of the letters $x$ and $y$ occurs in $\alf(\mathbf w_i)$ for any $i\ne s$.
In this case, it is easy to see that the identity $\mathbf w(\mathcal T)\approx \mathbf w^\prime(\mathcal T)$, where $\mathcal T:=\{x,y,t_1,\dots,t_m\}$, coincides (up to renaming of letters) with some identity of the form~\eqref{two letters in a block} in $\Sigma$.
Since each of the letters $x$ and $y$ forms at most one island in any block of $\mathbf w$ and $\mathbf w^\prime$, we have
\[
\mathbf O\{\mathbf w(\mathcal T)\approx \mathbf w^\prime(\mathcal T)\}=\mathbf O\{\mathbf w\approx \hat{\mathbf w}\},
\]
and we are done.
\end{proof}

\subsection{Exclusion identities for a number of certain varieties}
\label{subsec: exclusion identities}

\begin{lemma}
\label{L: satisfies a_{n+m}[rho]=a_n^{n+m}[rho] old}
Let $\mathbf V$ be a monoid variety satisfying $xyzxy\approx yxzxy\approx xyzyx$ and $x^2y\approx yx^2$ such that $M(xyx)\in\mathbf V$.
If, for any $(n,m)\in\hat{\mathbb N}_0^2$ and $\rho\in S_{n,m}$, the monoid $M(\mathbf a_{n,m}[\rho])$ does not lie in $\mathbf V$, then $\mathbf V$ satisfies $\Psi_1\cup\Psi_2$.
\end{lemma}

\begin{proof}
By~\cite[Lemma~4.8]{Gusev-23}, the variety $\mathbf V$ satisfies the identities $\mathbf a_{n,m}[\rho]\approx \mathbf a_{n,m}^\prime[\rho]\approx \mathbf a_{n,m}^{\prime\prime}[\rho]$ for any $(n,m)\in\hat{\mathbb N}_0^2$ and $\rho\in S_{n,m}$.
It is easy to see that if $k,\ell\in\mathbb N_0$ and $\pi\in S_{k+\ell}$, then there are $(n,m)\in\hat{\mathbb N}_0^2$ and $\rho\in S_{n,m}$ such that $\mathbf a_{k,\ell}[\pi]\approx \mathbf a_{k,\ell}^\prime[\pi]\approx \mathbf a_{k,\ell}^{\prime\prime}[\pi]$ is a consequence of $\mathbf a_{n,m}[\rho]\approx \mathbf a_{n,m}^\prime[\rho]\approx \mathbf a_{n,m}^{\prime\prime}[\rho]$.
The lemma is thus proved.
\end{proof}

For any $n,m\in\mathbb N_0$, $\rho\in S_{n+m}$ and $0\le p\le q\le n+m$, we put
\[
\mathbf a_{n,m}^{p,q}[\rho]:=\biggl(\prod_{i=1}^n z_it_i\biggr)\biggl(\prod_{i=1}^p z_{i\rho}\biggr)x\biggl(\prod_{i=p+1}^q z_{i\rho}\biggr)x\biggl(\prod_{i=q+1}^{n+m} z_{i\rho}\biggr)\biggl(\prod_{i=n+1}^{n+m} t_iz_i\biggr).
\]
Notice that $\mathbf a_{n,m}^{0,n+m}[\rho]=\mathbf a_{n,m}[\rho]$, $\mathbf a_{n,m}^{n+m,n+m}[\rho]=\mathbf a_{n,m}^{\prime}[\rho]$ and $\mathbf a_{n,m}^{0,0}[\rho]=\mathbf a_{n,m}^{\prime\prime}[\rho]$.

\begin{lemma}
\label{L: satisfies a_{n+m}[rho]=a_n^{n+m}[rho]}
Let $\mathbf V$ be a monoid variety satisfying the identities $xyzxy\approx yxzxy\approx xyzyx$ such that $M(x^2,xyx)\in\mathbf V$.
If, for any $(n,m)\in\hat{\mathbb N}_0^2$ and $\rho\in S_{n,m}$, the monoid $M(\mathbf a_{n,m}[\rho])$ does not lie in $\mathbf V$, then $\mathbf V$ satisfies either $\Psi_1$ or $\Psi_2$.
\end{lemma}

\begin{proof}
If the variety $\mathbf V$ satisfies the identity $x^2y\approx yx^2$, then the required claim follows from Lemma~\ref{L: satisfies a_{n+m}[rho]=a_n^{n+m}[rho] old}.
So, we assume below that $\mathbf V$ violates the identity $x^2y\approx yx^2$.
Then by the condition of the lemma and Lemma~\ref{L: M(W) in V}, the words $xyx$, $x^2y$ and $yx^2$ are isoterms for $\mathbf V$.

First, verify that $\mathbf V$ satisfies either 
\[
\hat{\Psi}_1:=\{xyxty\approx yx^2ty,\,ytxyx\approx ytyx^2\} \ \text{ or } \ \hat{\Psi}_2:=\{xyxty\approx x^2yty,\,ytxyx\approx ytx^2y\}.
\]
Arguing by contradiction, suppose that $\mathbf V$ violates some identity in $\hat{\Psi}_1$ and some identity in $\hat{\Psi}_2$.
Since $M(xyxty)\notin\mathbf V$, it follows from Lemma~\ref{L: M(W) in V} that $\mathbf V$ satisfies a non-trivial identity $xyxty\approx \mathbf w$ for some $\mathbf w\in \mathcal X^\ast$. 
Since the words $x^2y$ and $xyx$ are isoterms for $\mathbf V$,  $\mathbf w(y,t)=yty$, $\mathbf w(x,t)=x^2t$.
This is only possible when $\mathbf w\in\{x^2yty,yx^2ty\}$.
Thus, $\mathbf V$ satisfies one of the identities $xyxty\approx x^2yty$ or $xyxty\approx yx^2ty$.
By a similar argument, we can show that $\mathbf V$ satisfies one of the identities $ytxyx\approx ytyx^2$ or $ytxyx\approx ytx^2y$, and
\begin{itemize}
\item one of the identities $yx^2ty\approx xyxty$ or $yx^2ty\approx x^2yty$ whenever $M(yx^2ty)\notin \mathbf V$;
\item one of the identities $ytyx^2\approx ytxyx$ or $ytyx^2\approx ytx^2y$ whenever $M(ytyx^2)\notin \mathbf V$.
\end{itemize}
It follows that
\begin{itemize}
\item if $M(yx^2ty)\notin \mathbf V$, then $xyxty\approx yx^2ty$ holds in $\mathbf V$;
\item  if $M(ytyx^2)\notin \mathbf V$, then $ytxyx\approx ytyx^2$ holds in $\mathbf V$.
\end{itemize}
Therefore, since $\mathbf V$ violates some identity in $\hat{\Psi}_1$, either $M(yx^2ty)\in \mathbf V$ or $M(ytyx^2)\in \mathbf V$.
By a similar argument, we can deduce from the assumption that $\mathbf V$ violates some identity in $\hat{\Psi}_2$ that either $M(x^2yty)\in \mathbf V$ or $M(ytx^2y)\in \mathbf V$.
However, this contradicts the condition of the lemma because of one of the following routinely checked claims:
\begin{itemize}
\item the monoid $M(xyxty)$ belongs to the variety $\mathbf M(yx^2ty)\vee\mathbf M(x^2yty)$;
\item the monoid $M(z_1t_1xz_2z_1xt_2z_2)$ belongs to the variety $\mathbf M(yx^2ty)\vee\mathbf M(ytx^2y)$;
\item the monoid $M(z_1t_1xz_1z_2xt_2z_2)$ belongs to the variety $\mathbf M(ytyx^2)\vee\mathbf M(x^2yty)$;
\item the monoid $M(ytxyx)$ belongs to the variety $\mathbf M(ytyx^2)\vee\mathbf M(ytx^2y)$.
\end{itemize}
Thus, we have proved that $\mathbf V$ satisfies either $\hat{\Psi}_1$ or $\hat{\Psi}_2$. 
By symmetry, we may assume that $\hat{\Psi}_1$ holds in $\mathbf V$.

In view of the above, we have proved that there exists $r\in\mathbb N$ such that $\mathbf V$ satisfies the identity $\mathbf a_{r_1,r_2}[\rho]\approx \mathbf a_{r_1,r_2}^\prime[\rho]$ for all $(r_1,r_2)\in\hat{\mathbb N}_0^2$ and $\rho\in S_{r_1,r_2}$ with $r_1+r_2\le r$ (for instance, $r=1$). 
We are going to verify that an arbitrary $r$ possesses this property. 
Arguing by contradiction, we suppose that the mentioned claim is true for $r=1,2,\dots,k-1$ but is false for $r=k$.
Then $\mathbf V$ violates $\mathbf a_{n,m}[\rho]\approx\mathbf a_{n,m}^\prime[\rho]$ for some $(n,m)\in\hat{\mathbb N}_0^2$ and $\rho\in S_{n,m}$ such that $n+m=k$.

If $M(xzxyty)\notin \mathbf V$, then $\mathbf V$ satisfies $\sigma_3$ by Lemma~\ref{L: swapping in linear-balanced}.
Then we have a contradiction with the choice of $n,m$ and $\rho$ because
\[
\mathbf a_{n,m}[\rho]=\mathbf px\mathbf qx\mathbf r\stackrel{\sigma_3}\approx\mathbf px\mathbf a\mathbf bx\mathbf r\stackrel{\sigma_3}\approx\mathbf p\mathbf q_1x^2\mathbf q_2\mathbf r\stackrel{\hat{\Psi}_1,\sigma_3}\approx\mathbf p\mathbf q_1\mathbf q_2x^2\mathbf r\stackrel{\sigma_3}\approx\mathbf px^2\mathbf q\mathbf r=\mathbf a_{n,m}^\prime[\rho],
\]
where
\[
\begin{aligned}
&\mathbf p:=\biggl(\prod_{i=1}^n z_it_i\biggr),\ \mathbf q:=\biggl(\prod_{i=1}^{n+m} z_{i\rho}\biggr),\  \mathbf r:=\biggl(\prod_{i=n+1}^{n+m} t_iz_i\biggr),\\ 
&\mathbf q_1:=\mathbf q(z_1,\dots,z_n)\ \text{ and }\ \mathbf q_2:=\mathbf q(z_{n+1},\dots,z_{n+m}).
\end{aligned}
\] 
So, we may assume that $M(xzxyty)\in \mathbf V$.

According to Lemma~\ref{L: M(W) in V}, $\mathbf V$ satisfies a non-trivial identity $\mathbf a_{n,m}[\rho] \approx \mathbf a$.
Further, since $xzxyty$ is an isoterm for $\mathbf V$ by Lemma~\ref{L: M(W) in V}, we have $\mathbf a_x=(\mathbf a_{n,m}[\rho])_x$.
It follows from the fact that $xyx$, $x^2y$ and $yx^2$ are isoterms for $\mathbf V$ that $\mathbf a(x,t_n,t_{n+1})= t_nx^2t_{n+1}$.
Since the identity $\mathbf a_{n,m}[\rho]\approx \mathbf a$ is non-trivial, we may assume without any loss that $\mathbf a=\mathbf a_{n,m}^{p,q}[\rho]$ for some $0<p\le q\le n+m$.
Let 
\[
Z:=\{z_{1\rho},t_{1\rho},\dots,z_{p\rho},t_{p\rho},z_{(q+1)\rho},t_{(q+1)\rho},\dots,z_{k\rho},t_{k\rho}\}.
\]
Clearly, $(\mathbf a_{n,m}[\rho])_Z$ coincides (up to renaming of letters) with $\mathbf a_{c,d}[\pi]$ for some $(c,d)\in\hat{\mathbb N}_0^2$ and $\pi \in S_{c,d}$ such that $c+d=q-p$.
Since $\mathbf a_{c,d}[\pi]\approx \mathbf a_{c,d}^\prime[\pi]$ holds in the variety $\mathbf V$, this variety must satisfy $\mathbf a=\mathbf a_{n,m}^{p,q}[\rho] \approx \mathbf a_{n,m}^{q,q}[\rho]$. 
If $q<n+m$, then the set
\[
\Psi_1\cup\{\mathbf a_{n,m}[\rho](x,z_i,t_i)\approx \mathbf a_{n,m}^{q,q}(x,z_i,t_i)\mid(q+1)\rho\le i\le k\rho\}
\]
implies the identity $\mathbf a_{n,m}^{q,q}[\rho]\approx \mathbf a_{n,m}^{n+m,n+m}[\rho]=\mathbf a_{n,m}^{\prime}[\rho]$.
Therefore, $\mathbf V$ satisfies the identity $\mathbf a_{n,m}[\rho]\approx \mathbf a_{n,m}^{\prime}[\rho]$, contradicting the choice of $n,m$ and $\rho$.
So, we have proved that $\mathbf a_{n,m}[\rho]\approx \mathbf a_{n,m}^{n+m,n+m}[\rho]=\mathbf a_{n,m}^{\prime}[\rho]$ holds in $\mathbf V$ for any $(n,m)\in\hat{\mathbb N}_0^2$ and $\rho\in S_{n,m}$.
 
Finally, since, for any $s,s^\prime\in\mathbb N_0$ and $\pi\in S_{s+s^\prime}$, there are $(n,m)\in\hat{\mathbb N}_0^2$ and $\rho\in S_{n,m}$ such that $\mathbf a_{n,m}[\rho]\approx \mathbf a_{n,m}^\prime[\rho]$ implies $\mathbf a_{s,s^\prime}[\pi]\approx \mathbf a_{s,s^\prime}^\prime[\pi]$, the set $\Psi_1$ is satisfied by $\mathbf V$.
The lemma is proved.
\end{proof}

\begin{corollary}
\label{C: satisfies a_{n+m}[rho]=a_n^{n+m}[rho]}
Let $\mathbf V$ be a monoid variety satisfying $xyzxy\approx yxzxy\approx xyzyx$ and $\Phi_k$ for some $k\in\mathbb N$ such that $M(xyx)\in\mathbf V$.
If, for any $(n,m)\in\hat{\mathbb N}_0^2$ and $\rho\in S_{n,m}$, the monoid $M(\mathbf a_{n,m}[\rho])$ does not lie in $\mathbf V$, then $\mathbf V$ satisfies either $\Psi_1$ or $\Psi_2$.
\end{corollary}

\begin{proof}
If $M(x^2)\in \mathbf V$, then the required claim follows from Lemma~\ref{L: satisfies a_{n+m}[rho]=a_n^{n+m}[rho]}.
If $M(x^2)\notin \mathbf V$, then one can easily deduce from Lemma~\ref{L: M(W) in V} that $\mathbf V$ satisfies $x^2\approx x^3$.
In this case, $\Phi_2$ holds in $\mathbf V$.
Now Lemma~\ref{L: satisfies a_{n+m}[rho]=a_n^{n+m}[rho] old} applies.
\end{proof}

\section{Proof of the main results}
\label{sec: proof}

To start with, we establish the following auxiliary result, which is the ``common part'' of the proofs of Theorems~\ref{T: small=ACC} and~\ref{T: DCC}

\begin{lemma}
\label{L: ACC or DCC}
Let $\mathbf V$ be a subvariety of $\mathbf A_\mathsf{cen}$ that satisfies either the ACC or the DCC.
Then $\mathbf V$ is contained in one of the varieties $\mathbf P_n$, $\mathbf P_n^\delta$, $\mathbf Q_n$ or $\mathbf Q_n^\delta$ for some $n\in\mathbb N$.
\end{lemma}

\begin{proof}
It is well known that any subvariety of $\mathbf A_\mathsf{cen}$ satisfies $\Phi_n$ for some $n\in \mathbb N$  (see~\cite{Straubing-82} or~\cite[Theorem~7.2.5]{Almeida-94}).
Let $n$ be the least number such that $\mathbf V$ satisfies $\Phi_n$.
If $n=1$, then $\mathbf V\subseteq\mathbf P_1$, and we are done.
So, we may further assume that $n>1$.

Suppose that $M(xyx)\notin\mathbf V$.
According to Lemma~\ref{L: M(W) in V}, the variety $\mathbf V$ satisfies a non-trivial identity $xyx\approx \mathbf v$.
Since $n>1$, Lemma~\ref{L: x^n is an isoterm} implies that $x$ is an isoterm for $\mathbf V$.
Then $\mathbf v=x^syx^t$, where either $s\ge 2$ or $t\ge 2$.
By symmetry, we may assume that $s\ge 2$.
If $n=2$, then $\mathbf V$ satisfies $x^2y\approx xyx$ because
$$
xyx\approx x^syx^t \stackrel{x^2\approx x^3}\approx x^2yx^t \stackrel{x^2y\approx yx^2}\approx yx^{2+t}\stackrel{x^2\approx x^3}\approx yx^2 \stackrel{x^2y\approx yx^2}\approx x^2y.
$$
If $n>2$, then $(s,t)=(2,0)$ by Lemma~\ref{L: x^n is an isoterm}.
We see that the identity $x^2y\approx xyx$ is satisfied by $\mathbf V$ in either case.
Since the latter identity implies the identities $\sigma_2$ and $\sigma_3$, we have $\mathbf V\subseteq\mathbf Q_n$. 
So, we may further assume that $M(xyx)\in\mathbf V$.
Three cases are possible.

\smallskip

\noindent\textit{Case }1: $\mathbf N\subseteq\mathbf V$.
It is shown in~\cite[Theorem~1.1]{Gusev-19} that the variety $\mathbf M(xtyzxy)\vee\mathbf N$ does not satisfy both the ACC and the DCC.
Therefore, $M(xtyzxy)\notin \mathbf V$.
According to Corollary~\ref{C: non-small}, the variety $\mathbf M(xzxyty)\vee\mathbf N$ also violates both the ACC and the DCC.
Therefore, $M(xzxyty)\notin \mathbf V$.
Now Lemma~\ref{L: swapping in linear-balanced} applies, yielding that $\mathbf V$ satisfies $\sigma_2$ and $\sigma_3$.
Hence $\mathbf V\subseteq \mathbf Q_n$.

\smallskip

\noindent\textit{Case }2: $\mathbf N^\delta\subseteq\mathbf V$. 
In this case, the arguments dual to ones from Case~1 can show that $\mathbf V\subseteq \mathbf Q_n^\delta$.

\smallskip

\noindent\textit{Case }3: $\mathbf N,\mathbf N^\delta\nsubseteq\mathbf V$.
According to Proposition~\ref{P: non-small}(ii),(iii) and the statement dual to it, the variety $\mathbf V$ does not contain the monoids 
$$
M(\mathbf c_{n,m,n+m+1}[\pi]),\ M(\mathbf d_{n,m,n+m+1}[\pi]),\ M(\mathbf c_{n,m,0}[\tau])\ \text{ and }\ M(\mathbf d_{n,m,0}[\tau])
$$
for all $n,m\in\mathbb N_0$, $\pi\in S_{n+m,n+m+1}$ and $\tau\in S_{n+m}$.
Then Lemma~4.9 in~\cite{Gusev-23} and the dual to it imply that $\mathbf V$ satisfies $\Psi_3$.
Further, $M(\mathbf a_{n,m}[\rho])\notin\mathbf V$ for any $(n,m)\in\mathbb N_0^2$ and $\rho\in S_{n,m}$ by Proposition~\ref{P: non-small}(i).
Now Corollary~\ref{C: satisfies a_{n+m}[rho]=a_n^{n+m}[rho]} applies, yielding that $\mathbf V$ satisfies either $\Psi_1$ or $\Psi_2$.
Hence either $\mathbf V\subseteq\mathbf P_n$ or $\mathbf V\subseteq\mathbf P_n^\delta$.
\end{proof}

\begin{proof}[Proof of Theorem~\ref{T: DCC}]
\textit{Necessity} follows from Lemma~\ref{L: ACC or DCC}. 

\smallskip

\textit{Sufficiency}.
By symmetry, it suffices to show that the varieties $\mathbf P_n$ and $\mathbf Q_n$ satisfy the DCC.
It follows from~\cite[Theorem~11.1]{Lee-23} that the variety $\mathbf Q_n$ is \textit{hereditary finitely based} in the sense that every its subvariety can be defined by a finite set of identities.
This implies that $\mathbf Q_n$ satisfies the DCC.

Now we are going to show that the variety $\mathbf P_n$ satisfies the DCC as well.
To do this, it suffices to verify that every subvariety $\mathbf V$ of $\mathbf P_n$ may be defined within $\mathbf P_n$ by a finite set of identities. 
We will employ the general result by Mikhail Volkov~\cite{Volkov-90}. 
According to Proposition~\ref{P: Psi_1,Psi_3 subvarieties}, $\mathbf V$ is defined within $\mathbf P_n$ by some set $\Gamma_1$ of identities of the form~\eqref{one letter in a block}, where $r,e_0,f_0,\dots,e_r,f_r\in\mathbb N_0$, together with some set $\Gamma_2$ of identities of the form~\eqref{two letters in a block}, where $k,\ell\in\mathbb N_0$, $p,q\in\mathbb N$, $g_1,\dots,g_{k+\ell}\in \mathbb N_0$ and $a_1,\dots,a_{k+\ell}\in\{x,y\}$. 
According to Corollary~2 in~\cite{Volkov-90}, the set $\Gamma_1$ is equivalent to some finite set $\Gamma_1^\prime$ of identities.
Further, if at least one of the numbers $p,q,g_1,\dots,g_{k+\ell}$ is not less than $n$, then the identity~\eqref{two letters in a block} is a consequence of $\Phi_n$.
Since $\mathbf V$ satisfies $\Phi_n$, this allows us to assume that $p,q,g_1,\dots,g_{k+\ell}< n$ for each identity~\eqref{two letters in a block} in $\Gamma_2$.
Then both hand-sides of each identity in $\Gamma_2$ have the form $\mathbf w_0t_1\cdots t_r\mathbf w_r$, where
\[
\mathbf w_1,\dots, \mathbf w_r\in\mathcal W:=\{1,x^i,y^j,x^iy^j,y^jx^i\mid 1\le i,j<n\}.
\]
Since the set $\mathcal W$ is finite, we can apply the theorem and Observation~1 in~\cite{Volkov-90}, yielding that the set $\Gamma_2$ is equivalent to a finite set $\Gamma_2^\prime$ of identities.
Therefore, $\mathbf V=\mathbf P_n\{\Gamma_1^\prime,\Gamma_2^\prime\}$, as required.

\smallskip

Theorem~\ref{T: DCC} is thus proved.
\end{proof}

\begin{proof}[Proof of Theorem~\ref{T: small=ACC}]
The implication (i)~$\Rightarrow$~(ii) is obvious. 
It remains to prove the implications (ii)~$\Rightarrow$~(iii)~$\Rightarrow$~(i).

\smallskip

(ii) $\Rightarrow$ (iii). According to Lemma~\ref{L: ACC or DCC}, $\mathbf V$ is contained in one of the varieties $\mathbf P_n$, $\mathbf P_n^\delta$, $\mathbf Q_n$ or $\mathbf Q_n^\delta$ for some $n\in\mathbb N$. Further, in view of~\cite[Proposition~12.18]{Lee-23}, the variety $\mathbf M\left(\{xt_1x\cdots t_nx\mid n\in\mathbb N\}\right)$ violates the ACC. Hence $M\left(\{xt_1x\cdots t_nx\mid n\in\mathbb N\}\right)\notin \mathbf V$. Then, by~\cite[Lemma~13.11]{Lee-23}, the variety $\mathbf V$ satisfies the identity
\[
x\biggl(\prod_{i=1}^kt_ix\biggr)\approx x^{e_0}\biggl(\prod_{i=1}^kt_ix^{e_i}\biggr),
\]
where $k\ge n$ and $e_0,\dots, e_k \in\mathbb N_0$ with $e_j \ge n$ for some $j \in \{0,\dots, k\}$.
Since
\[
x^{e_0}\biggl(\prod_{i=1}^kt_ix^{e_i}\biggr)\stackrel{\Phi_n}\approx x^n\biggl(\prod_{i=1}^kt_i\biggr),
\]
this implies that $\mathbf V$ satisfies $\omega_k$. 
Hence $\mathbf V$ contains in one of the varieties $\mathbf R_k$, $\mathbf R_k^\delta$, $\mathbf S_k$ or $\mathbf S_k^\delta$.

\smallskip

(iii) $\Rightarrow$ (i). The variety $\mathbf S_n$ is small by~\cite[Corollary~12.12]{Lee-23}. 
It remains to verify that the variety $\mathbf R_n$ is small.
Evidently, if $g_0,\dots,g_r\in\mathbb N_0$ and $\sum_{i=0}^rg_i\ge n+1$, then 
\[
x^{g_0}\biggl(\prod_{i=0}^rt_ix^{g_i}\biggr)\stackrel{\Phi_n,\omega_n}\approx x^n\biggl(\prod_{i=0}^rt_i\biggr).
\]
It follows that every identity of the form~\eqref{one letter in a block} is equivalent modulo $\{\Phi_n,\omega_n\}$ to some identity of the same form with $\sum_{i=0}^re_i,\sum_{i=0}^rf_i\le n$.
Further, it is also evident that if $\occ_x(a_1^{g_1}\cdots a_{k+\ell}^{g_{k+\ell}})\ge n$ or $\occ_y(a_1^{g_1}\cdots a_{k+\ell}^{g_{k+\ell}})\ge n$, then the identity~\eqref{two letters in a block} is a consequence of $\omega_n$.
So, in view of the above and Proposition~\ref{P: Psi_1,Psi_3 subvarieties}, each subvariety of $\mathbf R_n$ can be defined within $\mathbf R_n$ by some of the following identities: \eqref{one letter in a block}, where $r,e_0,f_0,\dots,e_r,f_r\le n$ and $\sum_{i=0}^re_i,\sum_{i=0}^rf_i\le n$; and~\eqref{two letters in a block}, where $k,\ell\le 2n$, $g_1,\dots,g_{k+\ell}\le n$, $p,q\le n$ and $a_1,\dots,a_{k+\ell}\in\{x,y\}$.
Since there are only finitely many such identities, the variety $\mathbf R_n$ is small, as required.
\end{proof}

\end{document}